\documentclass[11pt,twoside]{amsart}

\usepackage{enumerate}
\usepackage[hyperindex=true,plainpages=false,colorlinks=false,pdfpagelabels]{hyperref}

\theoremstyle{plain}
\newtheorem{The}{Theorem}
\newtheorem*{The*}{Theorem}
\newtheorem{Pro}[The]{Proposition}
\newtheorem{Lem}[The]{Lemma}
\newtheorem{Cor}[The]{Corollary}
\newtheorem*{Cor*}{Corollary}

\theoremstyle{definition}
\newtheorem*{Def}{Definition}

\theoremstyle{remark} 
\newtheorem{Rem}[The]{Remark}
\newtheorem{Exa}{Example}
\newtheorem*{Rem*}{Remark}

\numberwithin{equation}{section}

\DeclareMathOperator{\End}{End}
\DeclareMathOperator{\Hom}{Hom}
\DeclareMathOperator{\GL}{GL}

\DeclareMathOperator{\Id}{Id}

\DeclareMathOperator{\Hess}{Hess}

\DeclareMathOperator{\mH}{\mathcal H}
\DeclareMathOperator{\mJ}{\mathcal J}
\DeclareMathOperator{\ed}{d}
\DeclareMathOperator{\vol}{vol}

\DeclareMathOperator{\grad}{grad}
\DeclareMathOperator{\delr}{\frac{\del}{\del r}}
\DeclareMathOperator{\delphi}{\frac{\del}{\del \varphi}}

\DeclareMathOperator{\del}{\partial}

\newcommand{\R}{\mathbb{R}}
\newcommand{\C}{\mathbb{C}}

\newcommand{\Z}{\mathbb{Z}}

\begin{document}

\title[Conformally Flat Circle Bundles over Surfaces]{Conformally Flat Circle Bundles over Surfaces}

\author{Sebastian Heller}

\address{Sebastian Heller\\
  Mathematisches Institut\\
  Universit{\"a}t T{\"u}bingen\\\
  Auf der Morgenstelle 10\\
  72076 T{\"u}bingen\\
  Germany}

\email{heller@mathematik.uni-tuebingen.de}

\subjclass{53C12,53C24,53C43}

\date{\today}

\thanks{The Author would like to thank his supervisor Ulrich Pinkall.}

\thanks{Author supported by GRK 870 ''Arithmetic and Geometry'' and SFB/Transregio 71}

\begin{abstract} 
 We classify conformally flat Riemannian $3-$manifolds  which possesses a free isometric $S^1-$action.
 \end{abstract}

\maketitle

\section{Introduction}
\label{sec:intro}
We consider compact oriented conformally flat Riemannian $3-$manifolds $P$ such that there exists a free isometric circle action $P\times S^1\to P.$ The space of orbits has an unique structure of an oriented
Riemannian manifold $M=P/S^1$ such that the projection $\pi$ is a Riemannian submersion. Then
$S^1\to P\to M$ can be viewed as a principal $S^1-$bundle.
The orthogonal complement $\mathcal H$ of the vertical space $\mathcal V=\ker d\pi$ is a principal connection. The condition that $P$ is conformally flat can be written as a differential equation on $M$ in terms of the curvature function $H$ of this connection and the Gaussian curvature $K$ of $M:$
\begin{equation}
\begin{split}
&\Hess H=H(H^2-K)\Id\\
&2K-3H^2=\alpha,\\
\end{split}
\end{equation}
for some constant $\alpha.$ These equations are strongly related to the geometry of the surface.
For example we prove that around a regular critical point $p$ of $H,$ $H$ itself depends only on the geodesic distance $r=d(.,p)$ from $p.$ It satisfies $cH'(r)=L(r),$ where $c$ is some constant and $L(r)$ is the length of the geodesic circle of distance $r$ around $p.$ Using this we show that the curvature functions $H$ and $K$ must be constant. We give a full classification of conformally flat circle bundles.

\section{Circle Bundles}
First, we shortly describe the spaces we are considering and the situations in which circle bundles occur. Then we recall some formulas for the curvature on a fibered $3-$manifold. For details see \cite{B}.
We use these to compute the absolute exterior derivative of the Schouten tensor, i.e. the obstruction to $(P,g)$ being conformally flat.

We need some basics from the theory of principal bundles, see \cite{KN} or more briefly \cite{Fr}. 
\begin{Def}
A Riemannian manifold $(P,g),$ together with a submersion $\pi\colon P\to M,$ is called circle bundle with circle metric if it is a principal bundle $S^1\to P\to M$ such that $S^1$ acts by isometries.
\end{Def}
For short, we will say that $P$ is a circle bundle.
The proposition below characterizes Riemannian manifolds $P$ occurring as total spaces of circle bundles.

\begin{Pro}
A compact, oriented Riemannian $3-$manifold $P$ for which 
a Riemannian submersion $\pi\colon P\to M$ to an oriented surface $M$ with connected minimal fibers does exist possesses a free isometric circle action $P\times S^1\to P,$ and vice versa.
\end{Pro}
\begin{proof}
It is a basic property of the mean curvature of a submanifold that it measures the derivative of the volume. In our case of a Riemannian submersion from a Riemannian $3-$manifold to a surface  with compact, connected and minimal fibers, these fibers are all of the same length and diffeomorphic to $S^1.$ 
We assume the length to be $2\pi.$ Let $T$ be the vector field of constant length $1$
in positive fiber direction and consider its flow $\Phi.$ Clearly,
every integral curve of $T$ is closed with period $2\pi,$ hence the flow
induces  a free and proper circle action $$P\times S^1\to P;\ \ (x,e^{i\varphi})\mapsto\Phi_\varphi(x).$$ The minimality of the fibers is given by $\nabla_TT=0,$ which obviously implies that the metric is invariant under the
action.

Conversely, the space of orbits $P/S^1=M$ has an unique structure as a Riemannian manifold $(M,h)$ such that the canonical projection is a Riemannian submersion. The metric $h$ is given by the formula $$h(\ed\pi(X),\ed\pi(Y))=g(X,Y)$$ for horizontal vectors $$X,Y\in\mathcal H:=ker\ed\pi^\perp.$$
The property of $S^1$ acting isometrically on $P$ translates to the minimality of the fibers. 
\end{proof}

We are going to compute the Levi-Civita connection and the curvature of the
total space $P.$ From now on let $T$  be the vector field in positive
fiber direction of constant length $1.$ And let $\hat A,\hat B$ be the horizontal lifts of orthonormal fields $A,B$ of the base $M.$
Then there is a function $\lambda$ such that $e^{\lambda}(A+iB)$ is a holomorphic field on the surface. This condition is equivalent to $[A,B] = B \cdot \lambda A - A \cdot \lambda B.$

The $S^1-$action leaves the metric invariant. Therefore 
the horizontal distribution given by $\mathcal H=\ker\ed\pi^\perp$ is also invariant and gives rise to a principal connection.  The curvature $\Omega$ of this connection is an imaginary
valued $2-$form. It is invariant under $S^1,$ since $S^1$ is abelian.
Therefore, the function $H$ defined by 
\begin{equation}\label{defH}
\Omega=iH\pi^*\vol_M
\end{equation}
is constant along the fibers. Thus it generates a function on the surface, which will be denoted by 
$H,$ too. We denote by $K$ the Gaussian curvature (function) of the surface $M$ with respect to the metric $h.$
\begin{Rem}\label{unitvectors}
There is another possibility to obtain a total space $P$ with $S^1$- invariant metric. Take a complex unitary line bundle $\pi\colon L\to M$ with unitary connection $\nabla$ over a surface $M$ with Riemannian metric $h.$
Let $P$ be the set of unit length vectors in $L.$ Then $S^1\subset\C$ acts on $P$ by scalar multiplication. The
connection $\nabla$ gives rise to a horizontal space $HL\subset TL$ by declaring a germ of a section $[s]$ over a germ of a curve $[\gamma]$ in $M$ to be horizontal if and only if $s$ is parallel along $\gamma$ with respect to $\nabla.$ This induces a splitting $TL=VL\oplus HL,$ where
$VL:=\ker \ed\pi$ is isomorphic to $\pi^*L$ in a canonical way and $HL$ is isomorphic to $\pi^*TM$ via $\ed\pi.$ Therefore we obtain a Riemannian metric on the total space $L$ from the inner product on $L$ and the Riemannian metric on $M.$ We equip $P$ with the induced submanifold metric $g$ which is invariant under $S^1$ by construction. In fact each circle bundle $P$ with circle metric $g$ can be obtained in this way: Take the 
induced metric $h$ on $M$ and the associated unitary line bundle $L:=P\times_\rho\C$ with the induced unitary connection, where $\rho\colon S^1\subset\C^*\to\GL(\C)$ is the standard representation. Then we obtain $P$ by the construction described above. 
\end{Rem}
\begin{Rem}\label{orientation}
Note that a fixed orientation on the surface and a fixed $S^1-$action induces an unique orientation on the total space $P.$ If one wants to change the orientation of $P,$ one has to take the inverse $S^1-$action 
$(p,e^{i\phi})\mapsto pe^{-i\phi}.$ Then the associated unitary  line bundles and connections are dual to each over. Vice versa, let $L$ and $L^*$ be dual unitary line bundles with compatible unitary connections. Then the corresponding total
spaces with induced Riemannian metric are isometric with opposite orientation.
\end{Rem}
For stating the formulas below we will use the endomorphism  $\mJ\in\End(TP)$ given by $X\mapsto T\times X,$ where $\times$ is the cross-product on the oriented Riemannian $3-$space $(P,g).$ In fact $(\mH,\mJ)$ define a CR-structure on $P,$ such that $\ed\pi\colon \mH\to TM$ is complex linear.

With these notations it is simply a matter of computation to obtain:
\begin{Pro}\label{LCC}
Let $\hat A,\hat B,T,\lambda, H$ be defined as above. 
The Levi-Civita connection of the total space of a circle bundle 
$P\to M$ with circle metric 
$g$ 
 is given by the following equations,
\begin{equation*}
\begin{split}
\nabla T & =\frac{1}{2}H\mJ \\
\nabla A&=B\otimes g(.,\frac{1}{2}HT+\mJ\grad(\lambda\circ\pi))
+T\otimes g(.,\frac{1}{2}HB)\\
\nabla B&=-A\otimes g(.,\frac{1}{2}HT+\mJ\grad(\lambda\circ\pi))
-T\otimes g(.,\frac{1}{2}HA).\\
\end{split}
\end{equation*}
\end{Pro}

It is well-known that in dimension $3$ the Riemannian curvature tensor $R$ is entirely given by the Ricci tensor $Ric.$
Sometimes it is more convenient to work with the so-called Schouten tensor
$$S:=Ric-\frac{1}{4}scal \Id\in\End(TP)$$
instead of the Ricci tensor. For example we have
$$R=-S\cdot g,$$
where $\cdot$ is the Kulkarni-Nomizu product, and where we consider all tensors to be bilinear or multilinear forms, respectively.

With Proposition \ref{LCC} the curvature is given by
\begin{Pro}
The Schouten tensor of the total space of a circle bundle $P\to M$ with
circle metric $g$ is given by
\begin{equation*}\label{Sch}
\begin{split}
S(T,T)=&(-\frac{1}{2}K+\frac{5}{8}H^2)\\
S(T,X)=&g(-\frac{1}{2}\mJ\grad H,X)\\
S(X,Y)=&(-\frac{3}{8}H^2+\frac{1}{2}K)g(X,Y),\\
\end{split}
\end{equation*}
where $X$ and $Y$ are arbitrary horizontal vectors and $T,H,$ and $K$ are defined as above.
\end{Pro}

\section{Conformally Flat Circle Bundles}
A Riemannian manifold $(P,g)$ is conformally flat if there exists a local conformal diffeomorphism into the sphere equipped with its round metric around each point. This is equivalent to the existence of a locally defined function $\lambda$ such that $e^{2\lambda}g$ is flat, see
\cite {KP} or \cite{HJ} for more details.

In case of dimension $2,$ every metric is conformally flat. This is due to the fact that a metric together with an orientation give rise to an almost complex structure, which is already a complex structure  in dimension $2.$ Therefore there exist holomorphic charts, which are of course conformal.

For dimensions $n\geq4$ $P$ is conformally flat if and only if  the Weyl tensor $W,$ defined as the reminder in the general curvature decomposition $$R=-S\cdot g+W,$$ vanishes. The condition in dimension $3$ is of a higher order: Consider the Schouten tensor $S\in\End(TP)$ as a $TP-$valued $1-$form on $P.$ The Levi-Civita connection on $P$ gives rise to the absolute exterior derivative 
$$d^\nabla\colon\Omega^k(P;TP)\to\Omega^{k+1}(P;TP).$$
Then a metric is conformally flat if and only if 
\begin{equation}
d^\nabla S=0.
\end{equation}

Using the formulas for the Levi-Civita 
connection and the Schouten tensor in the propositions \ref{LCC} and \ref{Sch} and the respective notations one easily computes
\begin{equation}
\begin{split}
g(d^\nabla\mathcal S(\hat A\wedge\hat B),T)&=
\frac{1}{2}(-\Delta H+2H(K-H^2)),\\
g(d^\nabla\mathcal S( A\wedge T), B)&=
\frac{1}{2}(\Hess H( \hat A,\hat  A)-H(H^2-K)),\\
g(d^\nabla\mathcal S(B\wedge T),A)&=
\frac{1}{2}(\Hess H(\hat B,\hat B)-H(H^2-K)),\\
g(d^\nabla\mathcal S(\hat B\wedge T),\hat B)&=
-g(d^\nabla\mathcal S(A\wedge T),A)=
\frac{1}{2}\Hess H(\hat A,\hat B).\\
g(d^\nabla\mathcal S(A\wedge B),B)&=
-g(d^\nabla\mathcal S(A\wedge T),T)=
\hat A\cdot(\frac{1}{2}K-\frac{3}{4}H^2)\\
g(d^\nabla\mathcal S(B\wedge A), A)&=
-g(d^\nabla\mathcal S(B\wedge T),T)=
\hat B\cdot(\frac{1}{2}K-\frac{3}{4}H^2).\\
\end{split}
\end{equation}
These equations are defined on the
total space $P.$ But as the fibers are minimal and $H$ and $K$ are constant along
every fiber,  these terms are actually equations on the base. We obtain
\begin{The}\label{cfe}
A circle metric $g$ of a circle bundle $\pi\colon P\to M$ 
over a surface $M$ with Gaussian curvature $K$ and  curvature function $H$ of the horizontal distribution
is conformally flat if and only if
\begin{equation*}\label{hess}
\begin{split}
\Hess H&=\nabla\grad H=H(H^2-K)\Id\\
\end{split}
\end{equation*}
and
\begin{equation*}\label{KHeq}
\begin{split}
2K-3H^2&=\alpha\\
\end{split}
\end{equation*}
are satisfied on $M$ for some constant $\alpha.$
\end{The}
\begin{Rem}
The trace of the first equation has the same shape as the Willmore equation 
$$\Delta H+2H(H^2-K)=0$$
of an immersed surface in $\R^3$ with mean curvature function $H$ and Gaussian curvature function $K.$
\end{Rem}

\section{The Equations of Conformally Flatness}
We are going to illustrate how the equations of being conformally flat reduces to
an ODE here.

In the following we set $U:=\{x\in M\mid\grad H\neq0\}.$
\begin{Lem}\label{fundamental}
Given a function $H$ on a surface $M$ with Riemannian metric satisfying  $\Hess H=f\Id.$  
Then, for each point $p$ with $\grad_pH\neq0,$ there exist an open set
$V\subset U$ and a conformal chart $(x,y)\colon V\to\R^2$ such that the
metric $g$ on $V\subset M$ and the function $H$ does only depend on $x.$
\end{Lem}
\begin{proof}
Define on $U$ the vector fields $$X=\frac{1}{\parallel\grad H\parallel}\grad H$$ 
and $Y=\mJ X.$ Here $\mJ$ is the complex structure induced by $g$
and by the orientation, i.e.  $X,Y$ should form a positive oriented
orthonormal basis.
We set $$h=\ed H(X)$$ which is equivalent to $\grad H=hX.$ Using 
$\Hess H=f\Id$ and $\parallel X\parallel =1$ we get
\begin{equation}\label{nablaX}
\begin{split}
f&=\ed h(X),\\
0&=\ed h(Y),\\
\nabla X&=\frac{f}{h}Y\otimes g(.,Y),\\
\nabla Y&=-\frac{f}{h} X\otimes g(.,Y).\\
\end{split}
\end{equation}
We claim the existence of a locally defined function 
$l\colon V'\subset U\to \R$ 
such that $[lX,lY]=0.$ In this case $lX,lY$ is the Gaussian basis fields of a 
conformal chart $(x,y)\colon V\subset V'\to\R^2.$ 
With 
\begin{equation*}
\begin{split}
[X,Y]&=\nabla_XY-\nabla_YX=-\frac{f}{h}Y\\
\end{split}
\end{equation*} 
we get
\begin{equation*}
\begin{split}
Y\cdot f&=Y\cdot X\cdot h=X\cdot Y\cdot h-[X,Y]\cdot h=0\\
\end{split}
\end{equation*}
and
\begin{equation*}
\begin{split}
[lX,lY]&=l^2[X,Y]-(Y\cdot l)lX+(X\cdot l)lY\\
&=(-\frac{f}{h}l+(X\cdot l))lY-
(Y\cdot l)lX.\\
\end{split}
\end{equation*}
We denote by $\omega_1=g(.,X)$ and 
$\omega_2=g(.,Y)$ the dual basis of $X$ and $Y.$ 
Then $[lX,lY]=0$ is equivalent
to 
\begin{equation*}
\begin{split}
\ed l&=l\frac{f}{h}\omega_1.\\
\end{split}
\end{equation*}
As $\ed\frac{f}{h}\omega_1=0$ there is  a solution $q$ of $\ed q=\frac{f}{h}\omega_1$
on each simply connected open set 
$V'\subset U.$  Then $l=e^q$ is a nowhere vanishing 
solution of $\ed l=l\frac{f}{h}\omega_1.$ 
We obtain a conformal chart $(x,y)\colon V\subset V'\to\R^2$ with Gaussian 
basis fields $lX,lY.$ The metric is given by
$$g=l^2(dx\otimes dx+dy\otimes dy),$$ which only depends on $x,$ since
$\frac{\del l}{\del y}=l(Y\cdot l)=0.$ 
\end{proof}

\begin{Rem}\label{chart}
From the proof of this lemma we also get the existence of a chart 
$(x,y)\colon V\to\R^2$ with Gaussian basis fields $X,lY,$ with
$l$ satisfying $\ed l=\frac{f}{h}l\omega_1,$ such that the metric $g$ and the 
function $H$ only depend on $x.$  
\end{Rem}

We restrict our attention to the case where the function $f$ is
given by Theorem 
\ref{cfe}.
\begin{Lem}\label{criticalpoints}
Let $H$ be a non-constant solution of 
\begin{equation}\label{hess1}
\Hess H=H(H^2-K)=-\frac{1}{2}(H^3+\alpha H)
\end{equation}
for some constant $\alpha.$ Then every critical point $p$ of $H$, for which an
integral curve $\gamma$ of $\grad H$ with $p=\lim_{t\to\pm\infty}\gamma(t)$ does exist, is a regular critical
point. 
\end{Lem}
\begin{proof}
By assumption, there is a point $q\in U$ near $p$ 
with $H(q)\neq H(p),$ such 
that the integral curve of $\grad H$ (if $H(q)<H(p))$ or of $-\grad H$ (if
$H(q)>H(p))$ passing through $q$ is going to $p.$ Consequently, 
the integral curve of $X$ or $-X$ is a geodesic converging to $p,$ too.
Let $\gamma\colon [0;b]\to M$ be the geodesic with $\gamma(0)=q, 
\gamma'(0)=X(q)$ and $\gamma(b)=p.$ Consider the function
$H(x):=H\circ\gamma(x).$ Because of equations \ref{nablaX} and \ref{hess1} it satisfies the ODE
\begin{equation}\label{ode}
2H''(x)+H^3(x)+\alpha H(x)=0
\end{equation}
with final value $H'(b)=0.$ 
If $p$ would be a singular critical point of $H,$  we would have 
$H^3(b)+\alpha H(b)=0.$ By Picard-Lindel\"of, its only solution would be constant contradicting $H(p)\neq H(q).$
\end{proof}

The geodesic polar coordinates around $p\in M$ are defined to be the
composition of the inverse of the exponential map at $p$
and the Euclidean polar coordinates of $T_pM$ with respect to some orthonormal basis.
In the case of a surface, we denote these coordinates by 
$(r,e^{i\varphi}),$ where the image of the polar coordinates
is $\R^{>0}\times S^1.$

\begin{Pro}\label{criticalpoint}
Let $(r,e^{i\varphi})\colon V\to ]0;R[\times S^1$ be geodesic polar 
coordinates around a regular critical point $p$ of a function $H$ which solves
the equations in Theorem \ref{cfe}.
Then the metric is locally given by
\begin{equation}\label{polarmetric}
g=dr^2+(\frac{L}{2\pi})^2d\varphi^2.
\end{equation}
Moreover, $L$ and $H$ do only depend on $r$ with $L(r)=cH'(r)$ for some constant $c\neq0.$
\end{Pro}
\begin{proof}
First we show that the integral curves of 
$X=\frac{\grad H}{\parallel\grad H\parallel}$ and $Y=\mJ X$ near critical 
points coincide with the coordinate lines of a geodesic polar coordinate system.
We have proven that $H$ has only regular critical points. Thus we can assume
$p\in M$ to be a non-degenerate local minimum. There exists a
neighborhood $V\subset U\subset M$ of $p$ such that every integral
curve of $-X=-\frac{1}{\parallel\grad H\parallel}\grad H$ starting at a point 
$q\in V\setminus\{p\}$ goes to $p$ in finite time. As the integral
curves of $-X$ are geodesics by \ref{nablaX}, there is
a normal neighborhood $V$ of $p,$ i.e. a set which is diffeomorphic to an
open set in $T_pM$ via exponential map. 
Thus, every geodesic emanating from $p$ is an integral
curve of $X$ for small $t>0.$\\

We claim that the chart given by remark \ref{chart} is the same as the
geodesic polar coordinate system.
The proof of lemma \ref{criticalpoints} yields that the value
$H(q)$ for $q\in V$ depends only on the length of the integral curve of
$-X$ between $q$ and $p.$ And since these integral curves are geodesics
in a normal neighborhood its length is equivalent to the distance $d(p,q)$ between $p,q\in M.$ 
Altogether we have that the integral curves of $X$ and $Y$ are the coordinate lines of the
geodesic polar coordinates, and that $X=\delr.$\\

Let $r=d(.,p)\colon V\setminus\{p\}\to\R$ be the distance function
centered at $p.$
Consider the length function $$L(r)=\int_{\gamma_r}g(.,Y)$$ of
circles with radius $r$ around $p.$ By using the notations and results of lemma \ref{fundamental} we get
\begin{equation*}
\ed g(.,Y)=\ed\omega_2=\frac{f}{h}\vol_M. 
\end{equation*}
Applying Stokes theorem and the fact that $f$ and $h$ are constant along the circles around $p$ we obtain
\begin{equation}
L'=\frac{f}{h}L.
\end{equation}
This shows that the integrability factor $l$ with $[X,lY]=0$ 
and the length function $L$ are the same up to a constant. 
We may fix this constant to be $2\pi,$ i.e. $2\pi l=L.$
By remark \ref{chart}, $lY$ is a Killing field and therefore a Jacobi 
field along every integral curve of $X.$ Fix a geodesic $\gamma$ emanating
from $p.$ Then $lY(r):=(lY)\circ\gamma(r)$ has the same initial values
as the Jacobi field $\delphi.$ With $2\pi l=L$ and $\parallel Y\parallel=1$ we 
have $lY(0)=0.$ Moreover, $$lY'=\nabla_XlY=\frac{f}{h}lY,$$ and 
 together with $L'(0)=2\pi$ it implies $lY'(0)=\mJ\gamma'(0).$ This yields 
$\delphi=lY.$\\

It remains to show that there is a constant $c\neq0$ such that $l(r)=cH'(r).$ We have that $f$ and $h$ are locally given by $h(r)=H'(r)$ and 
$f(r)=h'(r)=H''(r).$ Thus all non-vanishing solutions of $$l'=\frac{f}{h}=H''/H'l$$ are of the shape
$cH'$ for a constant $c.$
\end{proof}
\begin{Cor}\label{regularcriticalpoints}
Every non-constant solution $H$ of equation \ref{hess1} on a compact surface has regular critical points only.
\end{Cor}
\begin{proof}
First we show the existence of a regular critical point. Note that by equation \ref{nablaX} a geodesic 
$\gamma$ through a point $q\in U$ with initial value $\gamma'(0)=X_q$ is an integral curve
of $X$ (in $U$). Since $M$ is complete, $\gamma$ is defined for all $t.$ But by definition of $X$ the
integral curve of $X$ is obviously not defined for all $t > 0$. Thus $\gamma$ does not stay
in $U$ for all time and there exists a $t_0>0$ with
$t_0:=\inf\{t>0\mid\gamma(t)\notin U\}.$ Evidently, $p:=\gamma(t_0)$ is a regular critical point by lemma \ref{criticalpoints}.\\

Let $q\in M$ be another critical point of $H,$ and $\gamma\colon[0;b]\to M$ be a geodesic from 
$p$ to $q.$ In proposition \ref{criticalpoint}  $\gamma$ is proven to be an integral curve of $X$ on the interval
$]0;a_1[,$ where $a_1:=\inf\{t\in0;b]\mid \grad_{\gamma(t)}H=0\}.$ By using lemma \ref{criticalpoints} again, 
$\gamma(a_1)$ is  a regular critical point. Thus $\gamma$ is again an integral curve of $X$
but on the interval $]a_1;a_2[,$ where $a_2$ is defined analogously to $a_1.$ As the number of regular critical points is finite, $q$ is to be reached after a finite number of steps by
an integral curve of $X.$ Thus by using lemma \ref{criticalpoints} a last time, $q$ is regular.
\end{proof}
\begin{Exa}
We end this section by giving an example of a conformally flat circle
bundle over a non-compact surface with non-constant curvature. Let 
$$M:=]1-\epsilon,1+\epsilon[\times S^1$$ for some $1>\epsilon>0,$ and
let  $P=M\times S^1$ be the circle bundle with projection 
$\pi\colon P\to M$ on the first factor. There are globally defined and commuting basis fields $\delr$ and $\delphi$ on $M$ with dual basis $\ed r$ and $\ed\varphi.$ We define a Riemannian metric on
$M$ by
$$g=\ed r\otimes\ed r+l^2(r)\ed\varphi\otimes\ed\varphi$$
where $l\colon M\to\R$ is a nowhere vanishing function which only
depends on $r.$ For any function $f\colon M\to\R$ we set
$f':=\frac{\del f}{\del r}.$
We
compute the Levi-Civita connection of $g$ with $X:=\delr$ and $Y:=\frac{1}{l}\delphi,$ then
$\nabla X=l'Y\otimes\ed\varphi$ and $\nabla Y=-l'X\otimes\ed\varphi.$ The Gaussian curvature of $M$ is given by $$K=-\frac{l''}{l}.$$
Let $H\colon M\to\R$ be a function which only depends on $r$ 
satisfying $\Hess H=H(H^2-K)\Id$ and $2K-3H^2=\alpha$ for some constant $\alpha.$ Using equation
\ref{nablaX} one easily obtains that there exists a constant $c$
such that $l=cH'.$ By putting these equations together we obtain the ODE
\begin{equation}\label{exa1}
H''=-\frac{1}{2}H^3-\frac{\alpha}{2}H.
\end{equation}
Vice versa, it is not difficult to verify that any function $H$ satisfying equation \ref{exa1}
satisfies the conditions above, too, with a constant $c$ such that $l=cH'.$
Therefore there exists a family of solutions to these equations. They are depending on the choice of $\alpha,$ 
$H(1)$ and $H'(1).$ For each solution $H$ the 
corresponding metric $g$ is determined
up to a constant $c$ via $l=cH'.$ \\

We proved the existence of a surface with 
Riemannian metric $g$ and Gaussian curvature $K$ together with a function
$H$ satisfying the equations of theorem \ref{hess}.  It remains
to show that for every solution $H$ there is a circle metric on
$\pi\colon P\to M$ such that the curvature function of the horizontal distribution is given by $H.$ To do so,  we define $T$ to be the
infinitesimal generator of the circle action on $P,$ i.e. if one uses the
product coordinates $(r,\varphi,t)$ on 
$P=]1-\epsilon,1+\epsilon[\times S^1\times S^1$ then $T=\frac{\del}{\del t}.$
Define the connection $1-$form of this circle bundle by
$$\omega=\ed t+\frac{c}{2}H^2\ed\varphi.$$
Evidently, it defines a principal connection on $P.$ Moreover,
the symmetric bilinear form $$\tilde g=\pi^*g+\omega\otimes\omega$$
is strictly positive definite and invariant under the circle action. Therefore $\tilde g$ is a circle metric. The curvature function $\tilde H$ of the horizontal
distribution is given by $$\ed\omega=\tilde H\pi^*\vol_M.$$
We have
\begin{equation*}
\begin{split}
\ed\omega&=HcH'\ed r\wedge\ed\varphi=H\pi^*\vol_M,
\end{split}
\end{equation*}
thus $\tilde H=H,$ and $(P,\tilde g)$ is in fact a conformally flat circle bundle over an oriented
surface. Using proposition \ref{Sch} the sectional curvature of
the horizontal distribution is given by
$$K-\frac{3}{4}H^2=\frac{\alpha }{2}+\frac{3}{4}H^2,$$
and is clearly non-constant unless $H$ is constant. 
\end{Exa}
\section{Classification over compact Surfaces}
We are going to classify conformally flat circle bundles over compact
oriented surfaces. 
Let $M$ be a compact surface. 
With the
same notations as in lemma \ref{fundamental} we state the following:
\begin{Lem}
The integral curves of $Y$ are complete in $U.$ Moreover they are closed.
\end{Lem}
\begin{proof}
First we show completeness. Let $\gamma\colon[a,b[\to U$ be an
integral curve of $Y.$ Since $M$ is compact there exists a sequence 
$(t_n)$ with $t_n<b$ and $t_n\rightarrow b$ such that $\gamma(t_n)\rightarrow p\in M.$ But
$\gamma'$ is of constant 
length $1,$ thus $\gamma(t)\rightarrow p$ for all 
$t\rightarrow b,\, t<b .$ Assume that 
$p\notin U,$ i.e. $\grad_pH=0.$  
In lemma \ref{fundamental} it was proven that 
$||\grad H(\gamma(t))||$ is constant for any integral curve of $Y,$ 
which is obviously not zero contradicting $\grad_pH=0.$\\

It remains to show that $\gamma\colon\R\to U\subset M$ is closed. If not, $\gamma$ would be injective.
As $M$ is compact, there would be a sequence $t_n\rightarrow\infty$
such that $\gamma(t_n)\rightarrow p$ for $n\rightarrow\infty.$ 
With the same arguments as above we get $p\in U.$ Since $\grad H\neq0$ on 
$U$ there is a neighborhood $V$ of $p$ such that $H(q)=H(p)$
if and only if $p$ and $q$ lie on the same integral curve of $Y.$ Because
$H$ is constant along $\gamma$ we have $H(\gamma(t_n))=H(p).$ Therefore,
$\gamma(t_n)$ and $p$ lie on the same integral curve of $Y$ for any $n$ large enough. By using $\parallel Y\parallel=1$ we have that $\gamma$
would pass $p$ infinitely often, which contradicts $\gamma$ to be injective.  
\end{proof}

\begin{Pro}\label{m=s}
Let $H\colon M\to\R$ be a solution of 
$$\Hess H=H(H^2-K)=-\frac{1}{2}(H^3+\alpha  H).$$
We have that either $H$ is constant or $M=S^2$ and there are exactly two critical points 
$N,S\in S^2$ of $H.$ In the second case
$S^2\setminus\{N,S\}\cong I\times S^1$ for some interval $I$ such that
the induced metric $g$ on $I\times S^1$ is given by
\begin{equation*}
g=dr^2+(\frac{L(r)}{2\pi})^2d\varphi^2,
\end{equation*}
where $(r,e^{i\varphi})$ are the product coordinates on $I\times S^1,$
and $L\colon I\to\R$ is the length function of circles in $S^2$
around $S.$  
Moreover, $H$ does only depend on $r$ satisfying $L(r)=cH'(r)$ for some constant
$c\neq0.$
\end{Pro}
\begin{proof}
Assume $H$ is not constant. This implies $M=S^2$ as a (simple) consequence of Morse
theory and the fact that the Hessian of $H$ at every critical point is strictly definite by 
$\Hess H=H(H^2-K)\Id.$\\

Let $S$ denote the absolute minimum and $N$ 
the absolute maximum of $H.$ Let $R$ be the geodesic distance from $S$ to $N.$
Fix a geodesic $\gamma\colon[0,R]\to S^2$ of shortest length from $S$ to $N$ and denote its 
parameter by $r.$ Consider the functions $H(r):=H\circ\gamma(r)$ and 
$h(r):=H'(r).$ They are related to the length $L(r)$ of the circles 
with radius $r$ around $S$ by $L(r)=c\ h(r)$ for some constant $c$ as shown in the proof of proposition \ref{criticalpoint}. 
Because of $L(R)=0$ every geodesic $\tilde\gamma$ with $\tilde\gamma(0)=S$ passes $N$ at time $R,$ i.e. $\tilde\gamma(R)=N.$
Therefore we have that $\exp_S$ is a diffeomorphism by restricting it
to $$B_R:=\{v\in T_SS^2\mid \parallel v\parallel<R\}.$$ The rest of the proof is obvious.
\end{proof}

\begin{The}\label{main}
Let $g$ be a conformally flat circle metric on a circle bundle 
$\pi\colon P\to M$ over a compact oriented surface $M.$
Then $M$ is of constant curvature $K,$ and $P$ is of constant curvature 
$H.$ Moreover, we have that $H=0$ or $M=S^2$ and we are in one of
the cases described below (\ref{constantdes}). 
\end{The}
\begin{proof}
For surfaces $M$ of genus $g\geq1$ we already know that $H$ is constant, see proposition \ref{m=s}. Therefore also $K$ is constant, and we are in the case of section \ref{constantdes}.

Assume that $H$ is non-constant, and let $S$ and $N$ be the absolute minimum 
and maximum of $H,$ respectively.
Let $$\gamma\colon[0;R]\to S^2;\ t\mapsto\gamma(t)$$
be a geodesic from $S$ to $N$ of minimal length. We have already 
shown that the equations $\Hess H=H(H^2-K)$ and 
$2K=3H^2+\alpha $ turn into
\begin{equation}\label{ode}
2H''(t)+H^3(t)+\alpha H(t)=0
\end{equation}
with $H'(0)=H'(R)=0,$ where we have used $H(t)=H\circ\gamma(t)$ for short.\\

We set $H(0)=A$ and $H(R)=B.$ Then every solution of 
equation \ref{ode} satisfies
\begin{equation}\label{conslaw}
\begin{split}
4(H'(t))^2=&-H^4(t)-2\alpha H^2(t)+2\alpha A^2+A^4\\
=&-(H(t)-A)(H(t)+A)(H^2(t)+2\alpha +A^2)
\end{split}
\end{equation}
by the law of conservation. Because $H$ is the
curvature function of a complex line bundle of degree $d$ and Gauss-Bonnet,
the integrals of $H$ and $K$ are given by
\begin{equation}\label{intH}
\begin{split}
-2\pi d=\int_{S^2}HdA&=c\int_0^RH(t)H'(t)dt=\frac{c}{2}(B^2-A^2)\\
4\pi=\int_{S^2}KdA&=\frac{c}{2}\int_0^R(3H^2(t)+\alpha )H'(t)dt=\frac{c}{2}
(B^3+cB-A^3-cA).
\end{split}
\end{equation}
Note that the sign in the first equation follows is caused by the 
definition of $H,$ see equation \ref{defH}.
Since $H'(R)=0,$ the equations \ref{conslaw} and \ref{intH} imply 
\begin{equation}\label{abeqn}
0=B^2+2\alpha +A^2.
\end{equation}
With $L'(0)=2\pi$ and equation \ref{ode} we obtain
\begin{equation}\label{leqn}
\begin{split}
0&=\frac{4\pi}{c}+A^3+\alpha  A,\\
0&=\frac{4\pi}{c}+B^3+\alpha  B.
\end{split}
\end{equation}
An easy algebraic computation reveals that the equations \ref{intH},
\ref{abeqn} and \ref{leqn} have no common solution.

Thus, $H$ must be constant. Again this implies that $K$ must be constant, too, and we are in one of the cases described below.
\end{proof}

\subsection{Classification of Circle Bundles with Constant Curvature $H$}\label{constantdes}
We end by describing the space of  conformally flat circle bundles over compact surfaces of genus $g$ with constant curvatures $H$ and $K.$ As we have seen in theorem \ref{main} these are the only conformally flat circle bundles with circle metric. The only condition to the constant functions $H$ and $K$ is 
$$H(H^2-K)=0.$$ 
We first describe the bundles over the sphere, and then the bundles over higher genus surfaces:

Let $M=S^2.$ From $H(H^2-K)=0$ we have either $H=0$ or $H^2=K.$  If $H=0,$ then $\pi\colon P\to S^2$ is the principal $S^1$-bundle of a flat complex unitary line bundle $L\to S^2.$ But $S^2$ is simply connected, thus
$L=S^2\times\C$ as a unitary bundle, and $P=S^2\times S^1$ with 
the product metric. In the case $H^2=K$ one easily sees that 
$H=-\frac{2}{d},$ where $d\neq0$ is the degree of the associate line 
bundle (by the standard representation 
$\rho\colon S^1\to\C=\GL(\C)$). Examples of these bundles are the
fibrations (induced by the Hopf fibration) of the lens spaces 
$$L(|d|;1)\to S^2,$$
where $L(|d|,1)$ is equipped with the metric of constant curvature $\frac{1}{d^2},$ and the sign of the degree $d$ depends on the chosen orientation. In fact, there are no other possibilities: If there were
another bundle $P$ of the same degree $d,$ then it would have
curvature $H=-\frac{2}{d},$ too. Then the induced connection of the
endomorpism bundle between $P\times_\rho\C$ and 
$L(|d|,1)\times_{\rho^{\pm1}}\C$ would have curvature $0,$ and 
would be again the trivial bundle $S^2\times\C.$ Thus these line bundles would be isometric.
The corresponding principal bundles with the circle metric can be
considered as the space of unit length vectors in the line bundles with 
induced metric, see remark \ref{unitvectors},so they are isometric, too.

For $g\geq1$ Gauss-Bonnet implies $K\leq0,$ and we get $H=0.$ Therefore
$P$ is the total space of a principal bundle with flat connection. There are many possibilities: First of all, for each Riemann surface (of genus $1$ or greater than $1$) there is exactly one metric in the given conformal class which has constant curvature $0$ or $-C$ for any $C>0,$ and all of these surfaces are clearly not isometric. And because of the non-trivial fundamental groups of surfaces with $g\geq1,$ there are many principal $S^1$-bundles with flat connection. To classify them up to isomorphisms,
consider the associated line bundles $L=P\times_\rho\C$ with induced connection $\nabla,$ see remark \ref{unitvectors}. It is well known that 
the space of flat unitary connections on a degree $0$ line bundle modulo gauge transformations is a torus of dimension $2g:$ 
To see this note that two of them are gauge equivalent if and only if they have the same holonomy representation $H^\nabla.$ 
But $S^1$ is abelian, so $H^\nabla$ factors 
through $\pi_1(M)\to H_1(M,\Z)$ and we obtain
$$H\colon H_1(M,\Z)\to S^1.$$
The abelian Lie algebra $\mathfrak g=Harm^1(M;i\R)$ of imaginary valued harmonic $1-$forms $\omega$ can be consider as the Lie algebra of 
$G:=\Hom(H_1(M,\Z), S^1)$  with exponential map given by 
$$\omega\mapsto([\gamma]\mapsto e^{\int_\gamma\omega}).$$
The kernel of $\exp$ is the full lattice $\Gamma$ of rank $2g$ of integral harmonic forms $\omega$ with $$\int_\gamma\omega\in2\pi i\Z$$ for all
$[\gamma]\in H_1(M,\Z).$ Therefore $G=\mathfrak g/\Gamma.$ Conversely,
for any imaginary valued harmonic $1-$form $\omega$ the connection
$$\nabla:=d-\omega$$
is unitary and has $[\gamma]\mapsto e^{\int_\gamma\omega}$ as monodromy.

It remains to determine whether two unit vector bundle $P$ and $\tilde P$
corresponding to unitary line bundles with non gauge equivalent flat connections $\nabla$ and $\tilde\nabla$ over possibly different surfaces $M$ and $\tilde M$ can be isometric. 

This can only happen for dual connections $\nabla$ and $\tilde\nabla=\nabla^*$ (note that these connections correspond to different orientations on $P,$ see \ref{orientation}). For $g\geq 2$  one might see this  as follows: we have $K<0$ and proposition \ref{Sch} shows that the fiber direction $\pm T$ is given by the orthogonal complement
of the $2-$plane $\hat A\wedge\hat B$ with minimal sectional curvature.

\end{document}